\documentclass[12pt]{amsart}

\usepackage{verbatim}
\usepackage{fullpage}
\usepackage{amsfonts}

\usepackage{graphicx}
\usepackage{amsmath, amsthm, amssymb}
\usepackage{fancybox}
\newcommand{\swoosh}{\includegraphics[width=0.35in]{swoosh.pdf}}

\newcommand{\R}{\mathbb{R}}
\newcommand{\C}{\mathbb{C}}
\newcommand{\N}{\mathbb{N}}

\newcommand{\Hil}{\mathcal{H}}

\newcommand{\eps}{\varepsilon}

\DeclareMathOperator{\lspan}{span}

\DeclareMathOperator{\ran}{ran}

\providecommand{\abs}[1]{\lvert#1\rvert}
\providecommand{\norm}[1]{\lVert#1\rVert}

\newcounter{Theorem}

\numberwithin{equation}{section}
\numberwithin{Theorem}{section}

\theoremstyle{plain} 
\newtheorem{thm}[Theorem]{Theorem}

\newtheorem{lem}[Theorem]{Lemma}
\newtheorem{prop}[Theorem]{Proposition}

\theoremstyle{definition}
\newtheorem{defn}[Theorem]{Definition}

\theoremstyle{remark}

\newtheorem{ex}[Theorem]{Example}

\begin{document}

\title{Tight projections of frames on infinite dimensional Hilbert spaces}

\author{John Jasper}

\date{}

\begin{abstract} We characterize the frames on an infinite dimensional separable Hilbert space that can be projected to a tight frame for an infinite dimensional subspace. A result of Casazza and Leon states that an arbitrary frame for a $2N$-- or $(2N-1)$--dimensional Hilbert space can be projected to a tight frame for an $N$--dimensional subspace. Surprisingly, we demonstrate a large class of frames for infinite dimensional Hilbert spaces which cannot be projected to a tight frame for any infinite dimensional subspace.
\end{abstract}

\address{Department of Mathematics, University of Missouri, Columbia, MO 65211--4100, USA}

\email{jasperj@missouri.edu}

\thanks{The author was supported by NSF ATD 1042701}

\keywords{frames, projections of frames, compressions}

\subjclass[2010]{Primary: 46C05, Secondary: 47A20}
\date{\today}
\date{}

\maketitle

\section{Introduction}

A sequence of vectors $\{f_{i}\}_{i\in I}$ in a Hilbert space $\Hil$ is called a \textit{frame} for $\Hil$ if there exist constants $0<A\leq B<\infty$ such that
\begin{equation}\label{frameineq}A\norm{f}^{2}\leq\sum_{i\in I}\abs{\langle f,f_{i}\rangle}^{2}\leq B\norm{f}^{2}\end{equation}
for all $f\in\Hil$. The numbers $A$ and $B$ are called the \textit{frame bounds}. If $A=B$, then $\{f_{i}\}$ is called a {\it tight frame}. The \textit{frame operator} $S:\Hil\to\Hil$ is given by 
\[Sf = \sum_{i\in I}\langle f,f_{i}\rangle f_{i}\quad\text{for all }f\in\Hil.\]

The crucial property of frames that makes them useful in practice is their basis--like reconstruction formula. That is, given a frame $\{f_{i}\}_{i\in I}$ for a Hilbert space $\Hil$ and any $f\in\Hil$ we have $f=\sum_{i\in I} \langle f,S^{-1}f_{i}\rangle f_{i}$, where $S$ is the frame operator of $\{f_{i}\}$. Since it may be difficult to invert $S$, this reconstruction formula may be of little use. For this reason we often concentrate on tight frames. Indeed, if $\{f_{i}\}$ is a tight frame with frame bound $A$, then for every $f\in\Hil$ we have the simple reconstruction formula $f=A^{-1}\sum \langle f,f_{i}\rangle f_{i}$.

A common problem in frame theory can be stated as follows: given a frame $\{f_{i}\}$, find a tight frame that retains some of the structure of $\{f_{i}\}$. One example is the problem of scalable frames. Two recent papers \cite{ckl,kop} have given characterizations of frames $\{f_{i}\}$ such that $\{c_{i}f_{i}\}$ is a tight frame for some sequence of positive scalars $\{c_{i}\}$. Another example is frame completions, in which vectors are added to a frame so that the resulting set of vectors is a tight frame, see \cite{fmp,mrs}.

In \cite{cl} Casazza and Leon considered the problem of projecting a given frame onto a subspace such that the projected vectors form a tight frame for the subspace. They showed that if $\{f_{i}\}_{i=1}^{M}$ is a frame for a $2N$-- or $(2N-1)$--dimensional Hilbert space, then there exists a projection $P$ onto an $N$-dimensional subspace such that $\{Pf_{i}\}_{i=1}^{M}$ is a tight frame for the image of $P$. In this paper we will characterize the frames on an infinite dimensional separable Hilbert space which can be projected to a tight frame for an infinite dimensional subspace. Specifically, we will prove the following.

\begin{thm}\label{frame} Let $\{f_{i}\}_{i\in\N}$ be a frame for a separable infinite dimensional Hilbert space $\Hil$. There is a projection $P$ onto an infinite dimensional subspace $\Hil_{0}$ such that $\{Pf_{i}\}_{i\in\N}$ is a tight frame for $\Hil_{0}$ if and only if the frame operator of $\{f_{i}\}$ is not a translate of a compact operator whose positive or negative part has finite dimensional kernel.
\end{thm}

In the finite dimensional case, any frame can be projected onto a tight frame for a subspace of approximately half the dimension of the original space. Thus it is natural to expect that frames in infinite dimensional spaces can be projected onto a tight frame for an infinite dimensional subspace. Surprisingly, there are frames for infinite dimensional Hilbert spaces that cannot be projected onto a tight frame for any infinite dimensional subspace, see Example \ref{noproj}.

Also in \cite{cl} it was shown that there are frames in $2N$-- and $(2N-1)$--dimensional spaces that cannot be projected onto tight frames for any subspace of dimension larger than $N$. Since the projection in Theorem \ref{frame} is already onto an infinite dimensional subspace, for a ``larger'' subspace we look at those with finite codimension. In Theorem \ref{finker} we show that a frame can be projected onto a tight frame for a subspace with finite codimension if and only if the frame operator is already a multiple of the identity on a subspace with finite codimension. Thus, apart from these exceptional frames, the result in Theorem \ref{frame} is optimal.

Theorem \ref{frame} will follow from the following more general statement about operators.

\begin{thm}\label{oper} Let $E$ be a positive noncompact operator on a separable infinite dimensional Hilbert space $\Hil$. There is a projection $P$ onto an infinite dimensional subspace and a constant $\alpha> 0$ such that $PEP=\alpha P$ if and only if $E$ is not a translate of a compact operator whose positive or negative part has finite dimensional kernel.
\end{thm}

The paper is organized as follows. In Section 2 we show that Theorem \ref{frame} follows from Theorem \ref{oper}. In section 3 we prove Theorem \ref{oper} for diagonalizable operators. We also show the nonexistence of a projection $P$ for the exceptional operators in the statement of Theorem \ref{oper}. In Section 4 we prove Theorem \ref{oper} for nondiagonalizable operators. Combining the nondiagonalizable and diagonalizable statements, we prove Theorem \ref{oper}.

\section{Preliminaries}

\begin{defn} Let $E$ be a self-adjoint operator. There exist unique positive operators $E_{+}$ and $E_{-}$, called the {\it positive part} and {\it negative part} respectively, such that $E=E_{+}-E_{-}$. 
\end{defn}

In the remaining sections we must repeatedly refer to compact operators $K$ such that $\dim\ker(K_{+})<\infty$ or $\dim\ker(K_{-})<\infty$. Thus we reluctantly introduce the following notation.

\begin{defn} Let $\Hil$ be an infinite dimensional Hilbert space. Let $\mathcal{B}_{0}(\Hil)$ denote the compact operators on $\Hil$. Define the set
\[FK(\Hil) = \{K\in\mathcal{B}_{0}(\Hil): K=K^{\ast}\text{ and either } \dim\ker(K_{+})<\infty\text{ or }\dim\ker(K_{-})<\infty\}.\]
\end{defn}

If $E$ is self-adjoint and diagonalizable, then there is an orthonormal basis of eigenvectors $\{e_{i}\}_{i\in I}$ of $E$. Let $\{\lambda_{i}\}_{i\in I}$ be the corresponding eigenvalues. In this case, the positive and negative parts are given by
\[E_{+}f = \sum_{\{i:\lambda_{i}>0\}}\lambda_{i}\langle f,e_{i}\rangle e_{i}\quad\text{and}\quad E_{-}f = -\sum_{\{i:\lambda_{i}<0\}}\lambda_{i}\langle f,e_{i}\rangle e_{i}.\]
From this we see that a compact self-adjoint operator $K$ is in $FK(\Hil)$ if and only if it has either finitely many nonnegative or nonpositive eigenvalues (with multiplicity). Moreover, since $K$ is an operator on an infinite dimensional space, it has either infinitely many positive or negative eigenvalues (without multiplicity).

The main result of this section, Proposition \ref{impl}, shows that Theorem \ref{oper} implies Theorem \ref{frame}. We require the following standard fact from frame theory, see \cite{ch}.

\begin{prop}\label{recon} A sequence $\{f_{i}\}_{i\in I}$ is a tight frame for $\Hil$ with frame bound $\alpha$ if and only if for each $f\in \Hil$
\[f = \frac{1}{\alpha}\sum_{i\in I}\langle f,f_{i}\rangle f_{i}.\]
\end{prop}

\begin{prop}\label{impl} Let $\{f_{i}\}_{i\in I}$ be a frame for a Hilbert space $\Hil$ with frame operator $S$, and let $P$ be a projection. The sequence $\{Pf_{i}\}_{i\in I}$ is a tight frame for $P\Hil$ if and only if there is some $\alpha>0$ such that $PSP=\alpha P$.
\end{prop}

\begin{proof} First, we will show that $\{Pf_{i}\}$ is a frame for $P\Hil$. Let $f\in P\Hil$, so $Pf = f$. Then,
\begin{equation}\label{imp1}\sum_{i\in I}\abs{\langle f,Pf_{i}\rangle}^{2} = \sum_{i\in I}\abs{\langle Pf,f_{i}\rangle}^{2} = \sum_{i\in I}\abs{\langle f,f_{i}\rangle}^{2}.\end{equation}
Since $\{f_{i}\}$ is a frame, there exist constants $0<A\leq B<\infty$ so that the last expression in \eqref{imp1} is bounded above and below by $A\norm{f}^{2}$ and $B\norm{f}^{2}$ respectively. Thus the frame inequality \eqref{frameineq} holds for all $f\in P\Hil$. 

By Proposition \ref{recon}, $\{Pf_{i}\}$ is a tight frame for $P\Hil$ with frame bound $\alpha$ if and only if for each $f\in \Hil$ we have
\[\alpha Pf = \sum_{i\in I}\langle Pf,Pf_{i}\rangle Pf_{i} = P\left(\sum_{i\in I}\langle Pf,f_{i}\rangle f_{i}\right) = PSPf.\]
\end{proof}

\section{Diagonalizable operators}

We begin with a lemma that generalizes \cite[Theorem 2.3]{cl} to diagonalizable operators on an infinite dimensional space. See Lemma \ref{blockn} for another generalization. The proof is a straightforward modification of that in \cite{cl}. However, since it is short, we include it.

\begin{lem}\label{block} Let $E$ be a diagonalizable normal operator on a separable Hilbert space $\Hil$. Let $\{e_{i}\}_{i\in I}$ be an orthonormal basis of eigenvectors of $E$ with corresponding eigenvalues $\{\lambda_{i}\}_{i\in I}$. Let $\{\sigma_{j}\}_{j\in J}$ be disjoint subsets of $I$. For each $j\in J$ let $\{a_{i}\}_{i\in\sigma_{j}}$ be a sequence of scalars such that $\sum_{i\in\sigma_{j}}\abs{a_{i}}^{2} = 1$, and set $f_{j} = \sum_{i\in\sigma_{j}}a_{i}e_{i}$. Set $\Hil_{0}=\overline{\lspan}\{f_{j}\}_{j\in J}$, and let $P$ be the projection onto $\Hil_{0}$. Then, $\{f_{j}\}_{j\in J}$ is an orthonormal basis for $\Hil_{0}$, and for each $j\in J$ we have $PEPf_{j} = \eta_{j}f_{j}$, where $\eta_{j} = \sum_{i\in\sigma_{j}}\abs{a_{i}}^{2}\lambda_{i}$.
\end{lem}

\begin{proof} Since the supports of the $f_{j}$ are disjoint, it is clear that $\{f_{j}\}$ is an orthogonal system. The assumption that $\sum_{i\in\sigma_{j}}|a_{j}|^{2}=1$ implies $\norm{f_{j}}=1$ for each $j\in J$. Thus $\{f_{j}\}$ is an orthonormal basis for $\Hil_{0}$. Finally, for each $j\in J$ we have
\begin{align*}
PEPf_{j} & = PE\left(\sum_{i\in \sigma_{j}}a_{i}e_{i}\right) = P\left(\sum_{i\in \sigma_{j}}a_{i}Ee_{i}\right) = P\left(\sum_{i\in \sigma_{j}}a_{i}\lambda_{i}e_{i}\right) = \sum_{k\in J}\left\langle \sum_{i\in\sigma_{j}}a_{i}\lambda_{i}e_{i},f_{k}\right\rangle f_{k}\\
 & = \left\langle \sum_{i\in\sigma_{j}}a_{i}\lambda_{i}e_{i},f_{j}\right\rangle f_{j} = \left\langle \sum_{i\in\sigma_{j}}a_{i}\lambda_{i}e_{i},\sum_{i\in\sigma_{j}}a_{i}e_{i}\right\rangle f_{j} = \left(\sum_{i\in\sigma_{j}}\abs{a_{i}}^{2}\lambda_{i}\right)f_{j}.
\end{align*}
\end{proof}

\begin{thm}\label{diag} Let $E$ be a noncompact diagonalizable positive operator on a separable infinite dimensional Hilbert space $\Hil$. If $E$ is not a translate of an operator in $FK(\Hil)$, then there is a projection $P$ onto an infinite dimensional subspace such that $PEP=\alpha P$ for some $\alpha> 0$.
\end{thm}

\begin{proof} First, assume there is some $\alpha>0$ such that $\dim\ker(E-\alpha)=\infty$. If $P$ is the projection onto $\ker(E-\alpha)$, then $PEP=\alpha P$. Thus we may assume that $\dim\ker(E-\alpha)<\infty$ for all $\alpha>0$.

Let $\{e_{i}\}_{i=1}^{\infty}$ be an orthonormal basis of eigenvectors of $E$ with associated eigenvalues $\{\lambda_{i}\}_{i=1}^{\infty}$. We wish to find a number $\alpha$ such that
\begin{equation}\label{diag1}|\{i\in\N: \lambda_{i}<\alpha\}| = |\{i\in\N:\lambda_{i}\geq \alpha\}| = \infty.\end{equation}
Consider the set of limit points of $\{\lambda_{i}\}$. By limit point we mean a real number $x$ such that, for all $\eps>0$, the set $\{i\in\N:\lambda_{i}\in(x-\eps,x+\eps)\}$ is infinite. If $\{\lambda_{i}\}$ has two limit points $x$ and $y$ such that $x<y$, then we let $\alpha\in(x,y)$.  The positivity of $E$ implies $x\geq 0$. This shows $\alpha>0$, and it is clear that \eqref{diag1} holds. If $\{\lambda_{i}\}$ has only one limit point $x$, then $E-x$ is compact. Since $E$ is not compact $x>0$. By assumption $E-x\notin FK(\Hil)$ and $\dim\ker(E-x)<\infty$. We deduce that \eqref{diag1} holds for $\alpha=x$. In either case we have \eqref{diag1} for some $\alpha>0$.

Let $\{\lambda_{n_{j}}\}_{j=1}^{\infty}$ and $\{\lambda_{m_{j}}\}_{j=1}^{\infty}$ be the subsequences of terms $<\alpha$ and $\geq\alpha$ respectively. For each $j\in \N$ define the set $\sigma_{j} = \{n_{j},m_{j}\}$. Since $\lambda_{n_{j}}<\alpha\leq \lambda_{m_{j}}$ for each $j\in\N$, there exists $a_{n_{j}},a_{m_{j}}\in[0,1]$ such that $a_{n_{j}}^{2}\lambda_{n} + a_{m_{j}}^{2}\lambda_{m} = \alpha$ and $a_{n_{j}}^{2}+a_{m_{j}}^{2} = 1$. For each $j\in\N$ set $f_{j} = a_{n_{j}}e_{n_{j}} + a_{m_{j}}e_{m_{j}}$. By Lemma \ref{block}, if $P$ is the projection onto $\Hil_{0}=\overline{\lspan}\{f_{j}\}_{j=1}^{\infty}$, then $\{f_{j}\}_{j=1}^{\infty}$ is an orthonormal basis for $\Hil_{0}$ and each $f_{j}$ is an eigenvector of $PEP$ with eigenvalue $a_{n_{j}}^{2}\lambda_{n} + a_{m_{j}}^{2}\lambda_{m} = \alpha$. In other words $PEP = \alpha P$, as desired.\end{proof}

To finish this section we will prove the ``only if'' direction of Theorem \ref{oper}. That is, if $E$ is a translate of an operator in $FK(\Hil)$ then there is no infinite rank projection $P$ such that $PEP=\alpha P$. First, we need a lemma.

\begin{lem}\label{frk} Let $P$ and $K$ be operators on an infinite dimensional Hilbert space. Assume $P$ is a projection and assume $K$ is a positive operator with $\dim\ker(K)<\infty$. If $PKP$ is finite rank, then $P$ is finite rank.
\end{lem}

\begin{proof} Define the subspace $V = \ran P\cap \ker(PKP)$. Let $\{v_{i}\}_{i\in I\cup J}$ be an orthonormal basis for $\ran P$ such that $\{v_{i}\}_{i\in I}$ is an orthonormal basis for $V$. Set $W={\overline{\lspan}}\{v_{i}\}_{i\in J}$. For $v\in V$ we have $0=PKPv=PKv$. This implies $Kv\in\ker P = (\ran P)^{\perp}$ and thus $\langle Kv,v\rangle = 0$. Since $K$ is a positive operator we conclude that $v\in\ker K$ and thus $V\subset \ker K$. Since the kernel of $K$ is finite dimensional, so is $V$.

Assume toward a contradiction that $W$ is infinite dimensional, which is equivalent to $J$ being infinite. For each $w\in W\setminus\{0\}$ we have $PKPw\neq 0$. Since $\ran PKP$ is finite dimensional, the set $\{PKPv_{i}\}_{i\in J}$ is dependent. There is a finite subset $F\subset J$ and nonzero scalars $\{\beta_{i}\}_{i\in F}$ such that $\sum_{i\in F}\beta_{i}PKPv_{i} = 0$. However, the vector $\sum_{i\in F}\beta_{i}v_{i}$ is a nonzero vector in $W$. Thus $PKP\sum_{i\in F}\beta_{i}v_{i}\neq 0$. This contradiction shows $\dim W<\infty$. Since $\ran P = V\oplus W$ we have $\dim\ran P<\infty$.
\end{proof}

\begin{thm}\label{cpt} Let $\Hil$ be an infinite dimensional Hilbert space and let $K\in FK(\Hil)$. If $P$ is a projection onto an infinite dimensional subspace and $\alpha,\beta\in\R$, then $P(\beta+K)P\neq \alpha P$.
\end{thm}

\begin{proof} If there is some $\alpha,\beta\in\R$ and projection $P$ with $P(\beta+K)P = \alpha P$, then $PKP=(\alpha-\beta)P$. Thus, it is enough to show that $PKP\neq\alpha P$ for all projections $P$ and all $\alpha\in\R$.

Let $\{e_{i}\}_{i=1}^{\infty}$ be an orthonormal basis of eigenvectors of $K$ with associated eigenvalues $\{\lambda_{i}\}_{i=1}^{\infty}$. The positive and negative parts of $K$ are given by
\[K_{+}f = \sum_{\{i:\lambda_{i}> 0\}}\lambda_{i}\langle f,e_{i}\rangle e_{i}\quad \text{and}\quad K_{-}f = -\sum_{\{i:\lambda_{i}< 0\}}\lambda_{i}\langle f,e_{i}\rangle e_{i},\quad f\in\Hil.\]
By assumption, one of these operators has a finite dimensional kernel. We may assume without loss of generality that $\dim\ker(K_{+})<\infty$. Note that $K_{-}$ must be finite rank.

First, we consider the case that $\alpha=0$. Assume toward a contradiction that there is an infinite rank projection $P$ such that
\begin{equation}\label{cpt1}0=PKP = PK_{+}P - PK_{-}P.\end{equation}
Since $K_{-}$ is in the ideal of finite rank operators, $PK_{-}P$ is also finite rank. From \eqref{cpt1} we see that $PK_{+}P$ must also be a positive finite rank operator. Lemma \ref{frk} implies that $P$ is finite rank and gives the desired contradiction.

Next, assume there is some $\alpha>0$ and projection $P$ so that $PKP=\alpha P$. Since the compact operators form an ideal, we see that $PKP$ is compact. This implies that $P$ is compact and thus a finite rank projection. \end{proof}

\begin{ex}\label{noproj} The following is an example of a frame $\{\phi_{n}\}_{n=1}^{\infty}$ (in fact, a bounded orthogonal basis) for an infinite dimensional Hilbert space such that no projection of the frame onto an infinite dimensional subspace is a tight frame for the subspace.

Let $\Hil$ be an infinite dimensional Hilbert space with orthonormal basis $\{e_{n}\}_{n=1}^{\infty}$. For each $n\in\N$ set $\phi_{n} = (2-n^{-1})^{1/2}e_{n}$. The frame operator $S$ is given by
\[Sf = \sum_{n=1}^{\infty}\left\langle \sum_{m=1}^{\infty}\langle f,e_{m}\rangle e_{m},(2-n^{-1})^{1/2}e_{n}\right\rangle (2-n^{-1})^{1/2}e_{n} = \sum_{n=1}^{\infty}(2-n^{-1})\langle f,e_{n}\rangle e_{n}.\]
Thus $\{e_{n}\}_{n=1}^{\infty}$ is an orthonormal basis of eigenvectors of $S$ with associated eigenvalues $\{2-n^{-1}\}_{n=1}^{\infty}$. Define $K:\Hil\to\Hil$ by
\[Kf = \sum_{n=1}^{\infty}n^{-1}\langle f,e_{i}\rangle e_{i},\]
and note that $K\in FK(\Hil)$ and $S = 2-K$. Now, assume there is a projection $P$ onto an infinite dimensional subspace $\Hil_{0}\subset \Hil$ such that $\{P\phi_{n}\}_{n=1}^{\infty}$ is a tight frame for $\Hil_{0}$ with frame bound $\alpha>0$. That is, for $f\in \Hil$ we have
\[\alpha Pf = \sum_{n=1}^{\infty}\langle Pf,P\phi_{n}\rangle P\phi_{n} = P\left(\sum_{n=1}^{\infty}\langle Pf,\phi_{n}\rangle \phi_{n}\right) = PSPf.\]
From this, we see
\[PKP = P(2I-S)P = 2PIP-PSP = (2-\alpha)P.\]
By Theorem \ref{cpt} this is impossible.
\end{ex}

\section{Nondiagonalizable operators}

In this section we wish to extend Theorem \ref{diag} to nondiagonalizable operators.

\begin{defn} Let $\mu$ be a positive measure on $X$. Given $\phi\in L^{\infty}(X,\mu)$, the operator $M_{\phi}:L^{2}(X,\mu)\to L^{2}(X,\mu)$ given by
\[(M_{\phi}f)(x) = \phi(x)f(x),\quad x\in X,\quad f\in L^{2}(X,\mu)\]
is called the {\it multiplication operator} of $\phi$.
\end{defn}

We will use the following version of the Spectral Theorem \cite{conw}.

\begin{thm}\label{specthm} Let $N$ be a normal operator on a separable Hilbert space $\Hil$. There exists a $\sigma$-finite measure space $(X,\mu)$ and a function $\phi\in L^{\infty}(X,\mu)$ such that $N$ is unitarily equivalent to $M_{\phi}$.
\end{thm}

The following lemma is another generalization of \cite[Theorem 2.3]{cl}, this time adapted for multiplication operators.

\begin{lem}\label{blockn} Let $\mu$ be a positive measure on $X$ and let $\phi\in L^{\infty}(X,\mu)$. Let $\{X_{i}\}_{i\in I}$ be disjoint measurable subsets of $X$, each with positive measure. For each $i\in I$ let $f_{i}$ be a measurable function supported on $X_{i}$ such that $\norm{f_{i}}_{L^{2}(X,\mu)}=1$. Let $P$ be the projection onto $\Hil_{0} = \overline{\lspan}\{f_{i}\}_{i\in I}\subset L^{2}(X,\mu)$. Then, $\{f_{i}\}_{i\in I}$ is an orthonormal basis for $\Hil_{0}$, and for each $i\in I$ we have $PM_{\phi}Pf_{i} = \eta_{i}f_{i}$, where 
\[\eta_{i} = \int_{X_{i}}\phi\abs{f_{i}}^2\,d\mu.\]
\end{lem}

\begin{proof} Since the supports of the $f_{j}$ are disjoint, we see that $\{f_{j}\}_{j\in\N}$ is an orthonormal basis for $\Hil_{0}$. For $i\in I$ we have
\begin{align*}
PM_{\phi}Pf_{i} & = PM_{\phi}f_{i} = P\left(\phi\cdot f_{i}\right) = \sum_{j\in I}\langle \phi\cdot f_{i},f_{j}\rangle f_{j} = \langle \phi\cdot f_{i},f_{i}\rangle f_{i} = \left(\int_{X}\phi\cdot \abs{f_{i}}^2\,d\mu\right)f_{i}.
\end{align*}
Since $\phi\cdot f_{i}$ is supported on $X_{i}$, this gives the desired result.
\end{proof}

The following example demonstrates the use of Lemma \ref{blockn}.

\begin{ex} Let $E:L^{2}[0,1)\to L^{2}[0,1)$ be given by $(Ef)(x) = x\cdot f(x)$. Define $\{X_{i}\}_{i\in\N}$ by
\[X_{i} = [2^{-i-1},2^{-i})\cup[1-2^{-i},1-2^{-i-1}),\]
and set $f_{i}=2^{i/2}\chi_{X_{i}}$ for each $i\in\N$. We have
\[\norm{f_{i}}^{2} = \int_{0}^{1}\abs{f_{i}}^{2}\,dx = \int_{X_{i}}2^{i}\,dx = 2\cdot 2^{i}(2^{-i}-2^{-i-1}) = 1.\]
By Lemma \ref{blockn}, if $P$ is the projection onto $\Hil_{0} = \overline{\lspan}\{f_{i}\}_{i=1}^{\infty}$, then each $f_{i}$ is an eigenvector of $PEP$ with eigenvalue
\[\eta_{i} = \int_{X_{i}}x \abs{f_{i}(x)}^{2}\,dx = 2^{i}\int_{2^{-i-1}}^{2^{-i}}x\,dx + 2^{i}\int_{1-2^{-i}}^{1-2^{-i-1}}x\,dx = \frac{1}{2}.\]
Thus $PEP = (1/2)P$.
\end{ex}

\begin{lem}\label{part} Let $(X,\mu)$ be a $\sigma$-finite measure space, and let $\phi\in L^{\infty}(X,\mu)$ be a function which is not constant on any set of positive measure. If $y\in\C$ is in the essential range of $\phi$ then for any open set $B$ containing $y$ there is a countable infinite partition of $\phi^{-1}(B)$ into sets with finite positive measure. 
\end{lem}

\begin{proof} Under the $\sigma$-finiteness assumption it is enough to find a partition into sets of positive measure. Indeed, any set of infinite measure can be partitioned into countably many sets of finite positive measure.

Let $\{B_{n}\}_{n=1}^{\infty}$ be a nested sequence of open sets with $B_{1}=B$ and $\bigcap B_{n} = \{y\}$. Consider the sequence $a_{n}=\mu(\phi^{-1}(B_{n}))$. Since $y$ is in the essential range of $\phi$ we see that $a_{n}>0$ for all $n$. Moreover, we have
\[0 = \mu(\phi^{-1}(\{y\})) = \mu\left(\bigcap_{n=1}^{\infty} \phi^{-1}(B_{n})\right) = \lim_{N\to\infty}\mu\left(\bigcap_{n=1}^{N}\phi^{-1}(B_{n})\right) = \lim_{N\to\infty}a_{N}.\]
Thus, after passing to a subsequence (keeping the first term), we may assume $\{a_{n}\}_{k=1}^{\infty}$ is strictly decreasing. For each $n\in\N$ set
\[E_{n} = \phi^{-1}(B_{n}\setminus B_{n+1}).\]
Note that
\[\mu(E_{n}) = \mu(\phi^{-1}(B_{n}\setminus B_{n+1})) = \mu(\phi^{-1}(B_{n})) - \mu(\phi^{-1}(B_{n+1}))> \mu(\phi^{-1}(B_{n})) - \mu(\phi^{-1}(B_{n}))= 0,\]
and it is clear that $\{E_{n}\}$ is a partition of $\phi^{-1}(B)$.
\end{proof}

The next theorem is a version of Theorem \ref{oper} for operators with no eigenvalues.

\begin{thm}\label{cont} Let $(X,\mu)$ be a $\sigma$-finite measure space, and let $\phi\in L^{\infty}(X,\mu)$ be a nonnegative function which is not constant on any set of positive measure. There exists a projection $P$ onto an infinite dimensional subspace $\Hil_{0}\subset L^{2}(X,\mu)$ and $\alpha>0$ so that $PM_{\phi}P = \alpha P$.
\end{thm}

\begin{proof} Let $x$ and $y$ be distinct points in the essential range of $\phi$ with $x<y$. Let $B_{x}$ and $B_{y}$ be open balls with disjoint closures containing $x$ and $y$, respectively. By Lemma \ref{part}, there exist partitions $\{E_{n}\}_{n=1}^{\infty}$ and $\{F_{n}\}_{n=1}^{\infty}$ of $\phi^{-1}(B_{x})$ and $\phi^{-1}(B_{y})$ respectively, with $0<\mu(E_{n}),\mu(F_{n})<\infty$ for each $n\in\N$. For each $n\in\N$ set $a_{n} = \mu(E_{n})^{-1/2}$ and $b_{n}=\mu(F_{n})^{-1/2}$. Define the functions
\[f_{n} = a_{n}\chi_{E_{n}}\quad\text{and}\quad g_{n}=b_{n}\chi_{F_{n}},\]
where $\chi_{E}$ is the characteristic function of the set $E$. Let $P_{1}$ be the projection onto $\Hil_{1}=\overline{\lspan}\{f_{n},g_{n}\}_{n=1}^{\infty}$, and define $M_{1} = P_{1}M_{\phi}P_{1}$. By Lemma \ref{blockn}, for each $n\in \N$ we have $M_{1}f_{n} = \lambda_{n}f_{n}$ and $M_{1}g_{n} = \mu_{n}g_{n}$, where
\[\lambda_{n} = \frac{1}{\mu(E_{n})}\int_{E_{n}}\phi\,d\mu\quad\text{and}\quad \eta_{n} = \frac{1}{\mu(F_{n})}\int_{F_{n}}\phi\,d\mu.\]
For every $n\in\N$ we have $\phi(p)\in B_{x}$ and $\phi(q)\in B_{y}$ for almost all $p\in E_{n}$ and $q\in F_{n}$. This implies $\lambda_{n}\in B_{x}$ and $\eta_{n}\in B_{y}$ for every $n\in\N$.

The operator $M_{1}:\Hil_{1}\to\Hil_{1}$ is a diagonalizable operator on an infinite dimensional Hilbert space. Since $M_{1}$ has infinitely many eigenvalues (with multiplicity) in each of the disjoint closed intervals $\overline{B_{x}}$ and $\overline{B_{y}}$, it is not a translate of a compact operator. By Theorem \ref{diag} there is an infinite dimensional subspace $\Hil_{0}\subset \Hil_{1}$ and a constant $\alpha>0$ such that $QM_{1}Q=\alpha Q$, where $Q:\Hil_{1}\to\Hil_{1}$ is the projection onto $\Hil_{0}$. Letting $P:L^{2}(X,\mu)\to L^{2}(X,\mu)$ be the projection onto $\Hil_{0}$ yields $PM_{\phi}P=\alpha P$.
\end{proof}

\begin{proof}[Proof of Theorem \ref{oper}] First, assume that there is a projection $P$ with $\dim\ran P=\infty$ and a constant $\alpha> 0$ such that $PEP=\alpha P$. Theorem \ref{cpt} implies that $E$ is not a translate of an operator in $FK(\Hil)$.

Now, assume $E$ is a noncompact positive operator that is not a translate of an operator in $FK(\Hil)$. By the Spectral Theorem (Theorem \ref{specthm}) there is a $\sigma$-finite measure space $(X,\mu)$, a function $\phi\in L^{\infty}(X,\mu)$ and a unitary $U:\Hil\to L^{2}(X,\mu)$ so that $M_{\phi}=UEU^{\ast}$. Let $\sigma_{p}(E)$ be the set of eigenvalues of $E$, which is also the set of eigenvalues of $M_{\phi}$. Define the sets
\[X_{p}=\bigcup_{y\in\sigma_{p}(E)}\phi^{-1}(\{y\})\]
and $X_{c} = X\setminus X_{p}$. Since $L^{2}(X,\mu)$ is separable, the set $\sigma_{p}(E)$ is at most countable, and thus both $X_{p}$ and $X_{c}$ are measurable. Both $(X_{p},\mu)$ and $(X_{c},\mu)$ are $\sigma$-finite measure spaces. Let $\phi_{p}$ and $\phi_{c}$ be the restrictions of $\phi$ to $X_{p}$ and $X_{c}$, respectively. 

If $M_{\phi_{c}}:L^{2}(X_{c},\mu)\to L^{2}(X_{c},\mu)$ is the zero operator, then $M_{\phi}$ is unitarily equivalent to $M_{\phi_{p}}$. Since $M_{\phi_{p}}$ is diagonalizable, both $M_{\phi}$ and $E$ are also diagonalizable. In this case Theorem \ref{diag} gives the desired conclusion.

We may now assume that $M_{\phi_{c}}\neq 0$. By construction, $\phi_{c}$ is not constant on any set of positive measure. By Theorem \ref{cont} there is a projection $P_{c}$ onto an infinite dimensional subspace $\Hil_{c}\subset L^{2}(X_{c},\mu)$ and a constant $\alpha>0$ such that $P_{c}M_{\phi_{c}}P_{c} = \alpha P_{c}$. Let
\[\Hil_{0}= \{f\in L^{2}(X,\mu): f|_{X_{c}}\in \Hil_{c}\text{ and } f(x)=0\text{ a.e. }x\in X_{p}\}.\]
It is clear that $\Hil_{0}$ is infinite dimensional. If $P_{0}$ is the projection onto $\Hil_{0}$, then $P_{0}M_{\phi}P_{0} = \alpha P_{0}$. The operator $P=U^{\ast}P_{0}U$ is the projection onto the infinite dimensional subspace $U^{\ast}\Hil_{0}$ and
\[PEP = U^{\ast}P_{0}UEU^{\ast}P_{0}U = U^{\ast}P_{0}M_{\phi}P_{0}U = U^{\ast}(\alpha P_{0})U = \alpha P.\]

\end{proof}

In the proof of Theorem \ref{oper}, most of the projections we constructed have infinite dimensional kernel. To complete this paper we show that we may take the projection to have finite dimensional kernel if and only if $E$ is a translate of a finite rank operator. 

\begin{thm}\label{finker} Let $E$ be a positive operator on a Hilbert space $\Hil$, let $\alpha\geq 0$, and let $N\in\N$. There exists a projection $P$ with $\dim\ker P=N$ and $PEP=\alpha P$ if and only if $E-\alpha$ is a finite rank operator with $\dim\ran(E-\alpha)\leq 2N$.
\end{thm}

\begin{proof} First, assume that $E-\alpha$ is a finite rank operator with $\dim\ran(E-\alpha)\leq 2N$. If $P$ is the projection onto $\ker(E-\alpha)$, then $\dim\ker P = \dim\ran (E-\alpha)\leq 2N$ and $PEP=\alpha P$.

Now, assume that the projection $P$ exists. Define the subspace 
\[V = \{v\in\ran P:Ev=\alpha v\} = \ran P\cap \ker(E-\alpha).\]
Let $\{e_{i}\}_{i\in I\cup J}$ be an orthonormal basis for $\ran P$ such that $\{e_{i}\}_{i\in J}$ is an orthonormal basis for $V$. Set $W = \overline{\lspan}\{e_{i}\}_{i\in I}$. We have the orthogonal decomposition $\Hil = V\oplus W\oplus \ker P.$

First we show that $\dim W\leq \dim\ker P$. Assume toward a contradiction that $|I| = \dim W>\dim\ker P$. For each $i\in I$
\[\alpha e_{i} = \alpha Pe_{i} = PEPe_{i} = PEe_{i}.\]
This implies that for each $i\in I$ there is some $h_{i}\in \ker P$ such that $Ee_{i} = \alpha e_{i} + h_{i}$. The assumption that $|I|>\dim\ker P$ implies that the sequence $\{h_{i}\}_{i\in I}$ is dependent. There is a sequence of scalars $\{\beta_{i}\}_{i\in I}$, not all zero, such that $\sum_{i\in I}\beta_{i}h_{i} = 0$. Set $f = \sum_{i\in I}\beta_{i}e_{i}$. Since there is some $i\in I$ such that $\beta_{i}\neq0$, we see that $f\neq 0$. It is clear that $f\in W$. Next, we calculate
\[Ef = E\left(\sum_{i\in I}\beta_{i}e_{i}\right) = \sum_{i\in I}\beta_{i}(\alpha e_{i} + h_{i}) = \alpha\sum_{i\in I}\beta_{i}e_{i} = \alpha f.\]
This shows that $f\in V$, and thus $f$ is a nonzero vector in $V\cap W=\{0\}$. This contradiction shows that $\dim W\leq \dim\ker P$.

Next, let $y\in\ran(E-\alpha)$. There is some $x\in\Hil$ such that $(E-\alpha)x = y$. For any $v\in V$
\[\langle y,v\rangle = \langle(E-\alpha)x,v\rangle = \langle x,(E-\alpha)v\rangle = \langle x,0\rangle = 0.\]
This shows that $\ran(E-\alpha)\subseteq V^{\perp} = W\oplus\ker P$. Since we have already shown that $W$ is finite dimensional, this shows that $E-\alpha$ is finite rank. Finally, we have
\[\dim\ran(E-\alpha)\leq \dim V^{\perp}= \dim W+\dim\ker P\leq 2\dim\ker P.\]
\end{proof}

\end{document}